\documentclass[draftcls,onecolumn,12pt]{IEEEtran}   

\usepackage{graphicx}          
\usepackage{dsfont}
\usepackage{bbm}
\usepackage{amsmath}
\usepackage{graphics} 
\usepackage{epsfig} 
\usepackage{amssymb}  
\usepackage{subfigure}
\usepackage{setspace}
\usepackage{bbold,mathtools,color}
\usepackage{algorithm}
\usepackage{algorithmic}

\newenvironment{proof}[1][Proof]{\begin{trivlist}
\item[\hskip \labelsep {\bfseries #1}]}{\end{trivlist}}

\newtheorem{proposition}{Proposition}
\newtheorem{theorem}{Theorem}
\newtheorem{remark}{Remark}
\newtheorem{lemma}{Lemma}
\newtheorem{definition}{Definition}
\newtheorem{corollary}{Corollary}

\begin{document}


\title{ \bf \Large
 On the Complexity of the Constrained  Input Selection Problem for Structural Linear Systems}

\vspace{-0.9cm}

\author{\IEEEauthorblockN{  S\'ergio Pequito $^{\dagger,\ddagger}$ \qquad  Soummya Kar $^{\dagger}$ \qquad
A. Pedro Aguiar $^{\ddagger,\diamond}$ }

\thanks{ 
This work was partially supported by grant SFRH/BD/33779/2009, from Funda\c{c}\~ao para a Ci\^encia e a Tecnologia  (FCT),  the CMU-Portugal (ICTI) program, and NSF grant 1306128.

$^{\dagger}$ Department of Electrical and Computer Engineering, Carnegie Mellon University, Pittsburgh, PA 15213

$^{\ddagger}$  Institute for System and Robotics, Instituto Superior T\'ecnico, Technical University of Lisbon, Lisbon, Portugal

$^{\diamond}$ Department of Electrical and Computer Engineering, Faculty of Engineering, University of Porto, Porto, Portugal

}

}

\maketitle

\vspace{-0.9cm}

\begin{abstract}                          
This paper studies the problem of, given the structure of a linear-time invariant system and a set of possible inputs, finding the smallest subset of  input vectors that ensures system's structural controllability. We refer to this problem as  the \emph{minimum constrained  input selection} (minCIS) problem, since the selection has to be performed on an initial given set of possible inputs.  We prove that the minCIS problem is \mbox{NP-hard}, which addresses a recent open  question of whether there exist polynomial algorithms (in the size of the system plant matrices) that solve the minCIS problem.  To this end, we show that the associated decision problem, to be referred to as the CIS, of determining whether a subset (of a given collection of inputs) with a prescribed cardinality exists that ensures structural controllability, is NP-complete. Further, we explore in detail practically important subclasses of the minCIS obtained by introducing more specific assumptions either on the system dynamics or the input set instances for which systematic solution methods are provided by constructing explicit reductions to well known computational problems. The analytical findings are illustrated through examples in multi-agent leader-follower type control problems.

\end{abstract}


\section{Introduction}

Research on  large-scale control systems has grown considerably over the last few years, triggered by technological advances in sensing and actuation infrastructures and relatively low cost of deployment. Such pervasive sensing and actuation present tremendous opportunities for enhanced system control, although, at the cost  of handling and processing enormous amounts of sensor data for system state inference and subsequently coordinating generated control signals among the actuators distributed throughout the system. Thus,  it is of importance to understand  which subsets of  sensors and actuators (hence the \emph{smallest} amount of data that need to be processed and coordination  required) are crucial for achieving desirable system monitoring (observability) and control (controllability) performance. These and related questions form the core of the input/output selection problems~\cite{Magnus,largeScale,Skogestad04a} in large-scale control systems. In this paper, we focus on the problem of, given  a possibly large scale  linear-time invariant system and a set of possible inputs, finding the smallest subset of  input  vectors that ensures system's  controllability. Notice that, by duality between controllability and observability for linear-time invariant systems, another problem can be posed in terms of determining the minimal number of outputs that ensure  observability, whose solution is straightforward from knowing how to solve the related controllability problem.

Now, consider the system 
\begin{align}
\dot x(t)&=Ax(t) +B u(t) \label{input}
\end{align}
where $x\in \mathbb{R}^n$ is the state, $u\in \mathbb{R}^p$ and $y\in \mathbb{R}^m$ denote the input and output  vectors, respectively. Additionally, let $\bar A\in \{0,\star\}^{n\times n}$ denote the zero/nonzero or structural pattern of the system matrix $A$, whereas $\bar B\in\{0,\star\}^{n\times p}$  is   the structural pattern of the input matrix $B$; more precisely, an entry in these matrices is zero if the corresponding entry in the system matrices is equal to zero, and a free parameter (denoted by a star) otherwise. Notice that the structural matrices defined above determine the  coupling between the system state variables, and the  state variables actuated by the inputs  deployed  in the system. The structural  matrices are the object of study in \emph{structural systems theory}~\cite{dionSurvey}, where the pair $(\bar A,\bar B)$ is said to be \emph{structurally controllable} if there exists a numerical realization $(A,B)$ in \eqref{input} with the same structure, i.e., having zeros in the specified locations, as $(\bar{A},\bar{B})$ that is controllable.  
 In fact, a stronger \mbox{characterization} holds, and it can be shown that the set of non-controllable numerical realizations $(A,B)$ of a structurally controllable pair  $(\bar{A},\bar{B})$ has zero Lebesgue measure in the product space $\mathbb{R}^{n\times n}\times\mathbb{R}^{n\times p}$; in other words, \emph{almost all} numerical realizations of a structurally controllable pair are controllable~\cite{dionSurvey}. 
Hereafter, we restrict attention to structural system theoretic properties. More specifically, given the structural matrix and possible input configurations,  the \emph{minimum constrained  input selection} (minCIS) problem consists of  identifying the smallest subset of inputs that ensure structural controllability and may  be formally posed as follows

\noindent $\mathcal P_1$ Given $\bar A\in \{0,\star\}^{n\times n}$ and $\bar B\in \{0,\star \}^{n\times p}$, determine
\begin{align}
\mathcal J^*=\arg \min\limits_{\mathcal J\subset\{ 1,\ldots, p\}}& \qquad\qquad |\mathcal J|\\
\text{s.t. }\quad  & (\bar A ,\bar B_{\mathcal J}) \text{ is structurally controllable,}\notag
\end{align}

where $\mathcal J$ is a subset of indices associated with the inputs and $\bar B_{\mathcal J}$ corresponds to the subset of columns in $\bar B$ with index in $\mathcal J$.  \hfill $\diamond$

\begin{remark} The results that we obtain for the minCIS problem $\mathcal{P}_{1}$ readily extend to the corresponding output selection problem by the duality between observability and controllability in linear systems, and, hence, in what follows, we focus on the minCIS only. In addition,  note that the current setup considers continuous time systems, however, all our results apply to the discrete time setting  as well due to similar controllability criteria. \hfill $\diamond$
\end{remark}

Problem $\mathcal P_1$ has been previously explored by several authors, see \cite{CommaultD13} and references therein. In fact, \cite{CommaultD13} provided the motivation for the present paper,  in which the following question was posed: \emph{ Is there a polynomial solution to $\mathcal P_1$}?

In this paper, we address the above question in general scenarios.

%
%

{
In what follows, we use some concepts of computational complexity theory~\cite{Coo71}, that addresses  the classification of  (computational) problems into complexity classes. 
Formally, this classification is for \emph{decision problems}, i.e., problems with an ``yes" or ``no" answer.
 Further, for a decision problem, if there exists a procedure/algorithm that obtains the correct answer in a number of steps that is bounded by a polynomial  in the size of the input data  to the problem, then the algorithm is referred to as an \emph{efficient} or \emph{polynomial} solution to the decision problem and the decision problem is said to be polynomially solvable or belong to the class of polynomially solvable problems. A decision problem  is said to be in NP (i.e., the class of nondeterministic polynomially problems) if, given any possible solution instance, it can be verified using a polynomial procedure whether the instance constitutes a solution to the problem or not.  It is easy to see that any problem that is polynomially is also in NP, although, there are some problems in NP for which it is unclear whether polynomial solutions exist or not. These latter problems are referred to as being NP-complete. Consequently, the class of NP-complete problems are the \emph{hardest} among the NP problems, i.e., those that are verifiable using polynomial algorithms, but no polynomial algorithms  are known to exist that solve them.   Whereas the above classification is intended for decision problems, it can be immediately extended to optimization problems, by noticing that every optimization problem can be posed as a decision problem. More precisely, given a minimization problem, we can pose the following decision problem: Is there a solution to the minimization problem that is less than or equal to a prescribed value? On the other hand, if the solution to the optimization problem is obtained, then any decision version can be easily addressed. Consequently, if a (decision) problem is NP-complete, then the associated optimization problem is referred to as being NP-hard. We refer the reader to \cite{Garey:1979:CIG:578533} for an introduction to the topic, and Section~\ref{prelim} for further discussion.


In fact, one of the main results of the present paper consists in showing  the NP-completeness of the decision version of the minCIS  problem, which we refer to as \emph{constrained input selection} (CIS) problem, and given as follows.

$\mathcal P_1^d$ Is there a  collection of indices $\mathcal J\subset \{1,\ldots, p\}$ with at most $k$ elements (i.e., $|\mathcal J|\le k$) such that $(\bar A,\bar B_{\mathcal J})$ is structurally controllable?

The NP-completeness of CIS is attained by polynomially reducing the \emph{set covering problem} to it. Hence, in particular, polynomial complexity algorithms that solve general instances of the CIS and minCIS are unlikely to exist.  Nevertheless, there could be subclasses of the minCIS   that admit polynomial complexity algorithmic solutions, as  is the case with a practically relevant subclass of minCIS  problems identified in this paper; more precisely, when the input matrix $\bar{B}$ is restricted to be structurally similar to the $n\times n$ identity matrix\footnote{A structural input matrix $\bar{B}$ that is structurally similar to the $n\times n$ identity matrix is referred to as a \emph{dedicated input configuration}, in that, each input can actuate or is connected to at most a single state variable. Such dedicated input configurations are common in several large-scale multi-agent networked control systems such as the power system, see~\cite{IlicLe2008}, for example.} (but $\bar{A}$ is arbitrary).

In addition, since the CIS is NP-complete,  the minCIS may be polynomially reduced to other (more standard) NP-hard problems, through  polynomial reductions between their decision versions. Practically, such reduction may lead to efficient (polynomial complexity) approximation schemes for solving the minCIS with guaranteed suboptimality bounds. While we do not provide such reductions from general minCIS instances to other NP-hard problems, for a certain restricted subclass of minCIS problems (with some additional conditions on the dynamic matrix structure) we explicitly construct a reduction to the minimum set covering problem. This reduction builds upon the complexity remarks elaborated in \cite{CommaultD13}, yet it holds for a larger class of instances, and only relies  on a condition on the structure of the dynamics.} Furthermore, this restricted class is practically relevant and, as shown later, subsumes important applications in multi-agent control such as leader-follower problems \cite{MesbahiEgerstedt,rezatac07};  as a demonstration, we show how our reduction can be used to solve the  leader-selection problem and a more general variant of it, which we refer to as the  constrained leader-selection problem. 


{
The main results of the paper are threefold: (i) we show that CIS is  NP-complete, which implies that the minCIS is NP-hard; (ii) we identify a subclass of minCIS problems that are polynomially solvable; more precisely, under the assumption  that the input matrix is structurally similar to the identity matrix; and (iii) we provide a polynomial reduction of the
minCIS problem to a minimum set covering problem under a mild assumption  on the structure of the dynamic matrix (given in Assumption~1), that hold for several interconnected dynamical systems, as well as leader-selection problems like those introduced in Section 4.

}

The rest of this paper is organized as follows: Section~2 introduces some preliminaries on computational complexity theory, associated complexity classes and polynomial reductions between problems. Additionally,  we review some concepts and results in structural systems theory to be used in the sequel.  Section~3 presents the result that  the CIS is NP-complete, and, subsequently, minCIS is NP-hard.  In Section 4,  a polynomial reduction from the minCIS to the  minimum set covering problem is provided, under certain  assumptions on the minCIS instances. Finally,  an illustrative example is described in Section 5.

\section{Preliminaries and Terminology}\label{prelim}

In this section, we review the \emph{minimum set covering problem}, and its decision version,  referred to as the \emph{set covering problem}~\cite{Cormen}. In addition,   some necessary and sufficient conditions that ensure system's structural controllability, required to obtain the results presented in the paper, are introduced in Section~2.1.

A (computational) problem is said to be \emph{reducible in polynomial time} to another if there exists a procedure to transform the former to the latter using a polynomial number of operations on the size of its inputs. Such reduction is useful in determining the qualitative complexity class \cite{Garey:1979:CIG:578533} a particular problem belongs to. The following result may be used to check for NP-completeness of a given problem.

\begin{lemma}[\cite{Garey:1979:CIG:578533}]
If a problem $\mathcal P_A$ is NP-complete, $\mathcal P_B$ is in NP and $\mathcal P_A$ is  reducible in polynomial time to $\mathcal P_B$, then $\mathcal P_B$ is NP-complete.
\label{NPcompRed}
\hfill $\diamond$
\end{lemma}

Now, consider the  set covering (decision) problem: Given a collection of sets  $\{\mathcal S_j\}_{j=1,\ldots,p}$, where $\mathcal S_j\subset \mathcal U$, is there a collection of at most $k$ sets that covers $\mathcal U$, i.e.,  $\bigcup_{ j\in\mathcal K }\mathcal S_j =\mathcal U$, where $\mathcal K\subset\{1,\ldots,p\}$ and $|\mathcal K|\le k$? 

This is the decision problem associated with the \emph{minimum set covering problem}, a well known NP-hard problem, given as follows.

\begin{definition}[\cite{Cormen}]{(Minimum Set Covering Problem)}
\label{setcover}
Given a set of $m$ elements $\,\mathcal{U}=\left\{1,2,\hdots,m\right\}$ and a set of $n$ sets $\mathcal{S}=\left\{\mathcal S_1,\hdots,\mathcal S_n\right\}$ such that $\mathcal S_i\subset \mathcal U$, with $i\in \{1,\cdots,n\}$, and
$\displaystyle\bigcup_{i=1}^{n}\mathcal S_i=\mathcal{U}$, the  minimum set covering problem consists of finding a set of indices $\,\mathcal{I}^*\subseteq \left\{1,2,\hdots,n\right\}$ corresponding to the minimum number of sets covering $\mathcal U$, i.e.,

\[\begin{array}{rl}\mathcal{I}^*=\underset{\mathcal{I}\subseteq \left\{1,2,\hdots,n\right\}}{\arg \min}&\quad |\mathcal{I}|
\\
\text{s.t. } \quad &\mathcal{U}=\displaystyle\bigcup_{i\in\mathcal{I}}\mathcal S_i \ .
\end{array}\]

\hfill$\diamond$
\end{definition}

In particular, the set covering problem is used in the present paper to show the NP-completeness of $\mathcal P_1^d$, by considering the following result.

\begin{proposition}[\cite{Garey:1979:CIG:578533}]
Let $\mathcal P_A$ and $\mathcal P_B$ be two optimization problems, and $\mathcal P_B^d$ the decision versions associated with $\mathcal P_B$. 
If a problem $\mathcal P_A$ is NP-hard, an instance of $\mathcal P_B^d$ can be efficiently verified and $\mathcal P_A$ is  polynomially reducible to $\mathcal P_B$, then  $\mathcal P_B^d$ is NP-complete. In particular, $\mathcal P_B$ is NP-hard.\hfill $\diamond$\label{proposition1}
\end{proposition}

\subsection{Structural Systems}

Structural systems provide an efficient representation of a linear-time invariant system as a directed graph (digraph). A digraph consists of a set of \emph{vertices} $\mathcal V$ and a set of \emph{directed edges} $\mathcal E_{\mathcal V,\mathcal V}$ of the form $(v_i,v_j)$ where $v_i,v_j\in \mathcal V$. If a vertex $v$ belongs to the endpoints of an edge $e\in\mathcal E_{\mathcal V,\mathcal V}$, we say that the edge $e$ is incident to $v$. We represent the \emph{state digraph}  by $\mathcal{D}(\bar{A})=(\mathcal{X},\mathcal{E}_{\mathcal{X},\mathcal{X}})$, i.e., the digraph that comprises only the state variables as vertices denoted by $\mathcal X=\{x_1,\cdots,x_n\}$ and a set of directed edges between the state vertices denoted by $\mathcal{E}_{\mathcal X,\mathcal X}=\left\{(x_i,x_j)\in \mathcal X\times \mathcal X:\ \bar A_{j,i}\neq 0\right\}$. Similarly, we represent the \emph{system digraph} by $\mathcal D(\bar A,\bar B)=(\mathcal{X}\cup\mathcal{U},\mathcal{E}_{\mathcal{X},\mathcal{X}}\cup\mathcal{E}_{\mathcal{U},\mathcal{X}})$, where  $\mathcal U=\{u_1,\cdots,u_p\}$ corresponds to the input vertices and $\mathcal{E}_{\mathcal{U},\mathcal{X}}=\left\{(u_i,x_j)\in \mathcal U\times \mathcal X: \ \bar B_{i,j}\neq 0\right\}$ the edges identifying which state variables are actuated by which inputs. Further, we say that an input  $u_{i}$ is assigned to a state variable $x_{j}$ if $\bar{B}_{i,j}\neq 0$.

A \emph{directed path} between the vertices $v_1$ and $v_k$ is a sequence of edges $\{(v_1,v_2),(v_2,v_3),\hdots,$ $(v_{k-1},v_k)\}$. If all the vertices in a directed path are different, then the path is said to be an \emph{elementary path}. A \emph{cycle} is a directed path such that $v_1=v_k$ and all remaining vertices in the direct path are distinct.  


We also require the following graph theoretic notions~\cite{Cormen}: A digraph $\mathcal{D}$ is strongly connected if there exists a directed path between any two vertices. A \emph{strongly connected component} (SCC) is a maximal subgraph  $\mathcal{D}_S=(\mathcal V_S,\mathcal E_S)$ of $\mathcal{D}$ such that for every $u,v \in\mathcal V_S$ there exist paths from $u$ to $v$ and from $v$ to $u$.

By visualizing each SCC as a virtual node, we can build a \textit{directed acyclic graph} (DAG) representation, in which a directed edge exists between vertices belonging to two SCCs \emph{if and only if} there exists a directed edge connecting the corresponding SCCs in the original digraph $\mathcal D=(\mathcal V,\mathcal E)$. The construction of the DAG associated with $\mathcal{D}(\bar A)$ can be performed efficiently in $\mathcal{O}(|\mathcal V|+|\mathcal E|)$~\cite{Cormen}. {
In Figure~\ref{DAGexample}, we present a digraph and its DAG representation: by convention, the arrows connecting the different SCCs are facing downwards, which motivates the classification of the SCCs in the DAG as follows.}

\begin{definition}[\cite{PequitoJournal,PequitoACC13}]\label{linkedSCC}
 An SCC is said to be linked if it has at least one incoming/outgoing edge from another SCC. In particular, an SCC is \textit{non-top linked} if it has no incoming edges to its vertices from the vertices of another SCC.
\hfill $\diamond$
\end{definition}
\begin{figure}[ht]
\centering
\includegraphics[scale=0.45]{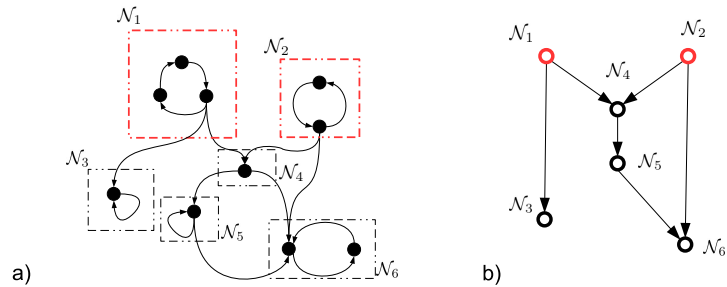}
\caption{ In a) the  SCCs are depicted by dashed boxes, labelled by $\mathcal N_i$ ($i=1,\ldots,6$), and  the non-top linked SCCs $\mathcal{N}_{1}$ and $\mathcal{N}_{2}$ are depicted in red. In b), these SCCs  correspond to vertices ($\mathcal{N}_{1}$ and $\mathcal{N}_{2}$) in the DAG representation.}
\label{DAGexample}
\end{figure}

Given $\mathcal D=(\mathcal V,\mathcal E)$, we can construct a \emph{bipartite graph} $\mathcal{B}(\mathcal S_1,\mathcal S_2,\mathcal E_{\mathcal S_1,\mathcal S_2})$, where $\mathcal S_{1}, \mathcal S_{2}\subset \mathcal V$ and the edge set $
\mathcal E_{\mathcal S_1,\mathcal S_2}=\{(s_1,s_2)\in \mathcal E \ :\ s_1 \in \mathcal S_1, s_2 \in \mathcal S_2  \ \}.
$ 
Such bipartite graphs will be used throughout in connection with the minCIS and we provide some elementary concepts associated with bipartite graphs.  Given $\mathcal{B}(\mathcal S_1,\mathcal S_2,\mathcal E_{\mathcal S_1,\mathcal S_2})$, a matching $M$ corresponds to a subset of edges in $\mathcal E_{\mathcal S_1,\mathcal S_2}$ that do not share vertices, i.e., given edges  $e=(s_1,s_2)$ and $e'=(s_1',s_2')$ with $s_1,s_1' \in \mathcal S_1$ and $s_2,s_2'\in \mathcal S_2$, $e, e' \in M$ only if $s_1\neq s_1'$ and $s_2\neq s_2'$.
A maximum matching $M^{\ast}$ is a matching $M$ with the largest number of edges among all possible matchings. Note that, in general, a maximum matching may not be unique. A maximum matching can be computed efficiently in $\mathcal{O}(\sqrt{|\mathcal S_1\cup \mathcal S_2|}|\mathcal E_{\mathcal S_1,\mathcal S_2}|)$ using, for instance, the Hopcroft-Karp algorithm \cite{Cormen}.

Given a matching $M$, an edge is said to be matched with respect to (w.r.t.) $M$, if it  belongs to $M$. In addition, we say that a vertex $\,v\in\mathcal{V}_1\cup\mathcal{V}_2\,$ is \emph{matched} if it is incident to some matched edge in $M$, otherwise we say that the vertex is \emph{free}  w.r.t. $M$. Incident and free vertices can be further characterized as follows: a vertex in $\mathcal S_2$ is a \textit{right-matched vertex} if it is incident to an edge in $M^{\ast}$, otherwise, it is an \textit{right-unmatched vertex}. A maximum matching in which there are no free vertices (or equivalently, either left/right-unmatched vertices) is called a \textit{perfect matching}.

Given a state digraph $\mathcal D(\bar A)=(\mathcal X,\mathcal E_{\mathcal X,\mathcal X})$, a particular bipartite graph of interest is its bipartite representation  denoted  as $\mathcal B(\bar A)\equiv \mathcal B(\mathcal X,\mathcal X,\mathcal E_{\mathcal X,\mathcal X})$, and we refer to it as the \emph{state bipartite graph}. The state bipartite graph may be used to characterize all possible structurally controllable pairs $(\bar A,\bar B)$, see~\cite{PequitoJournal}. In particular, in the sequel, we will use the following result.

\begin{proposition}[\cite{PequitoJournal,PequitoACC13}]
	Given $\mathcal{D}(\bar A)=(\mathcal X,\mathcal E_{\mathcal X,\mathcal X})$ and its DAG representation, constituted by $k$ SCCs, denoted by $\{\mathcal{N}_{i}\} _{i=1}^{k}$, where  $\mathcal{N}_{i}=(\mathcal{X}_i,\mathcal{E}_{\mathcal{X}_{i},\mathcal{X}_{i}})$, let $\mathcal{N}_{i_1},\hdots\mathcal{N}_{i_m}$ be the non-top linked SCCs in the DAG representation with $\{i_1,\hdots,i_m\}\subset\{1,\hdots,k\}$
and $\mathcal B(\bar A)$ the state bipartite graph. If $\mathcal B(\bar A)$ has a perfect matching, then $(\bar A,\bar B)$ is structurally controllable if and only if for each non-top linked SCC there exists an input  (corresponding to a column in $\bar{B}$) assigned to, i.e., connected to, at least one of its state variables. \hfill $\square$
\label{designPerfectMatching}
\end{proposition}


\vspace{-0.3cm}

\section{Main Results}


In this section, we show that the minCIS presented in $\mathcal P_1$ is NP-hard (Corollary~\ref{minCISNPhard}), by showing that its  decision version, the CIS, is an NP-complete problem (Theorem~\ref{NPcompCMIS}). Then, we identify a subclass of minCIS problems that are polynomially solvable (Theorem~\ref{subclassCMISPoly}).

 We start by showing  that CIS is NP-complete, as provided in the following result.

\begin{theorem}
The constrained  input selection  (CIS) problem presented in $\mathcal P_1^d$ is NP-complete.

\vspace{-0.4cm}
 \hfill $\diamond$
\label{NPcompCMIS}
\end{theorem}

\begin{proof}
The proof follows by using Proposition~\ref{proposition1}; more precisely, by presenting the polynomial reduction from the minimum set covering problem to minCIS, and noticing that $\mathcal{P}_{1}^{d}$ is in NP, i.e.,  there exist polynomial algorithms to verify if $(\bar A,\bar B(\mathcal J))$, for some $\mathcal J\subset \{1,\ldots,p\}$, is structurally controllable  \cite{CommaultD13}.

To obtain the polynomial reduction, consider a general minimum  set covering problem instance with sets $\{\mathcal S_i\}_{i\in \mathcal I}$, the index set $\mathcal I =\{1,\ldots, p\}$ and universe $\mathcal U=\bigcup\limits_{i\in\mathcal I} \mathcal S_i$, where $|\mathcal U|=n$. Subsequently, construct $\bar A\in\{0,\star\}^{n\times n}$ to be a diagonal matrix with nonzero entries, i.e.,  $\star$, in its diagonal. Additionally, select $\bar B \in \{0,\star\}^{n\times p}$, such that its $(i',j')$-th entry is given as follows:
\[
\bar B_{i',j'}=\left\{\begin{array}{cc}
\star,& \ \text{ if } i'\in \mathcal S_{j'}\\
0,& \text{ otherwise,}
\end{array}
\right.
\]
for $i'\in\{1,\ldots, n\}$ and $j'\in \{1,\ldots, p\}$.

 Note  that  such $\mathcal D(\bar A)=(\mathcal X,\mathcal E_{\mathcal X,\mathcal X})$,  consists of $n$ non-top linked SCCs and  the associated state bipartite graph $\mathcal{B}(\bar{A})$  has a perfect matching. Now, recall that, by Proposition~\ref{designPerfectMatching}, $(\bar{A},\bar{B}(\mathcal{J}))$, for some $\mathcal J\subset \{1,\ldots,n\}$,  is structurally controllable if and only if  each non-top linked SCC of $\mathcal{D}(\bar{A})$ contains a state variable that is connected from an input (corresponding to a nonzero column in $\bar{B}(\mathcal{J})$).

Subsequently, we first show that a feasible solution to the minCIS leads to a feasible solution of the minimum set covering problem, and secondly, a (minimal) solution to the minCIS leads to  a (minimal) solution of the minimum set covering problem. To show feasibility, let $\bar B(\mathcal J)$, for some $\mathcal J$, be a feasible solution to the minCIS, i.e., $(\bar A,\bar B(\mathcal J))$ is structurally controllable. It then follows that there exists  edges from the inputs associated with indices in $\mathcal J$ to all the state variables (corresponding to the non-top linked SCCs in $\mathcal{D}(\bar{A})$), which implies by the construction of $\bar B$ that the family of subsets $\{\mathcal S_j\}_{j\in\mathcal{J}}$   cover $\mathcal U$.

 To obtain minimality, suppose, on the contrary, that $\mathcal{J}^{\ast}$ constitutes a (minimal) solution to the minCIS, but the family $\{\mathcal{S}_{j}\}_{j\in\mathcal{J}^{\ast}}$ is not a minimum covering of $\mathcal{U}$.   Then, there exists $\mathcal J'\subset \{1,\ldots, p\}$ with $|\mathcal J'|<|\mathcal J^*|$ such that the family $\{\mathcal{S}_{j}\}_{j\in\mathcal{J}^{\prime}}$ covers $\mathcal{U}$. This, in turn, by the construction of $\bar{B}$ and Proposition~\ref{designPerfectMatching} implies that the pair $(\bar{A},\bar{B}(\mathcal{J}^{\prime}))$ is structurally controllable. Since  $|\mathcal J'|<|\mathcal J^*|$, we conclude that $\bar B(\mathcal J^*)$ is not a (minimal) solution to the minCIS, which is a contradiction.
\end{proof}

From Theorem \ref{NPcompCMIS}, we  obtain one of the main results of this paper.

\begin{corollary}\label{minCISNPhard}
 The minimum constrained  input selection (minCIS) problem is NP-hard.
\hfill $\diamond$
\end{corollary}


The fact that the minCIS is NP-hard,  however, does not rule out the possibility that there exist subclasses of the minCIS (with restricted input instances) that admit polynomial complexity algorithmic solutions (in the size of the system plant matrices).  In fact, a particularly interesting subclass of the minCIS  is  one in which the collection of inputs  initially given consist of all possible \emph{dedicated inputs}, i.e., the matrix $\bar{B}$  consists of $n$ inputs each of which is assigned to a single distinct state variable. Formally, we have the following result.

\begin{theorem}
Let $\bar A\in \{0,\star\}^{n\times n}$ be a given structural dynamic matrix   and $\bar{B}=\mathbb{I}_n\ $ a $\ n\times n$  diagonal input matrix with nonzero diagonal entries. The problem of determining  $\mathcal J^*$ such that

\begin{align}
\mathcal J^*=\arg\min\limits_{\mathcal J\subset \{1,\ldots,N\}}& \qquad\qquad |\mathcal J| \label{minDedInputSel}\\
\text{s.t. } \quad  (\bar A,&\mathbb{I}_n(\mathcal J)) \text{ is  structurally controllable,} \notag
\end{align}
where $\mathbb{I}_n(\mathcal J)$ corresponds to the columns of $\mathbb{I}_n$ with indices in $\mathcal J$,  referred to as the \emph{minimum dedicated input selection} problem,  can be solved polynomially. More precisely, in $\mathcal O(n^3)$.\hfill $\diamond$
\label{subclassCMISPoly}
\end{theorem}

\begin{proof} See Appendix. 
\end{proof}

In Theorem~\ref{subclassCMISPoly}, upon a restriction in $\bar B$, we obtained a subclass of minCIS problems that can be solved polynomially. Next, we impose some restrictions in $\bar A$, and we show that the problem can be systematically solved by resorting to a minimum set covering problem. 


\section{Partial Polynomial Reduction  of the minCIS to  the Minimum Set Covering Problem}

In Section 3 we have showed that $\mathcal P_1^d$ is an NP-complete problem  without explicitly deriving a polynomial reduction from $\mathcal P_1^d$ to an NP-complete problem, or equivalently, without  explicitly deriving a polynomial reduction from minCIS to another (standard or known) NP-hard problem. In this section, we provide a \emph{partial} polynomial reduction from the minCIS  to  the minimum set covering problem (see Theorem \ref{mainresult2} below). By  partial reduction we mean that it is only valid if the  state digraph satisfies certain additional properties, to be made precise in  Assumption~1. Notably,  the set of state digraphs satisfying Assumption 1 for which the proposed reduction holds, include dynamical systems commonly encountered in multi-agent networked control applications  (see Section 4.1 for details). Further, in Section 4.2 we show how the polynomial reduction obtained in Section 4.1 can be used to  solve leader-selection problems in multi-agent networks.

Throughout this section, we assume that the system dynamic matrices, i.e., the $\bar{A}$ matrices in the minCIS, satisfy the following condition.

\noindent \textbf{Assumption 1} The structural dynamic matrix $\bar{A}$ is such that the state bipartite graph   $\mathcal  B(\bar A)=\mathcal B(\mathcal X,\mathcal X,\mathcal E_{\mathcal X,\mathcal X})$ associated with $\bar{A}$, has a perfect matching. In other words, the set of right-unmatched vertices associated with any maximum matching of $\mathcal{B}(\bar{A})$ is empty.\hfill $\diamond$

\begin{remark}[\cite{PequitoJournal,PequitoACC13}]
In fact,  Assumption~1 can be interpreted in terms of the state digraph as follows: the state bipartite graph $\mathcal B(\bar A)$ has a perfect matching if and only if $\mathcal D(\bar A)$ is spanned by a disjoint union of cycles, or, alternatively,  it corresponds to a structural matrix such that \emph{almost all} of its numerical instances are full rank.\hfill $\diamond$
\end{remark}

We now provide a polynomial reduction from the minCIS to the minimum set covering problem under \mbox{Assumption~1}.

\begin{theorem}
Consider the minCIS problem with system matrix instance  $\bar A\in\{0,\star\}^{n\times n}$ and input matrix $\bar{B}\in\{0,\star\}^{n\times p}$, where $\bar{A}$ satisfies Assumption 1. Denote by  $\mathcal N^i$, $i=1,\ldots, k$, the $k$ non-top linked SCCs of $\ \mathcal{D}(\bar{A})$.  The minCIS problem  can then be polynomially reduced to the minimum set covering problem with universe  $\mathcal U= \{1,\ldots,k\}$ and   sets  $\{\mathcal S_j\}_{j=1,\ldots, p}$,  where $ \mathcal S_{j}=\{i\in \mathcal U : \ \bar B_{r,j}=\star ,\  x_r \in \mathcal N^{i} \}$.
\label{mainresult2}
\hfill $\diamond$
\end{theorem}

\begin{proof}
The proof requires two steps: 1) to show that the stated  reduction to the set covering problem can be achieved by performing a polynomial number of operations  with respect to the size of $\bar{A}$ and $\bar{B}$; and 2) to prove the correctness of the reduction, i.e., to show that, under Assumption~1, the solution to the minCIS can be readily determined from the minimal solution of the set covering problem.

The proposed  reduction is polynomial since the non-top linked SCCs of $\mathcal D(\bar A)$  can be determined polynomially, for instance,  by computing the DAG associated with $\mathcal D(\bar A)$ (see Section 2.1). Subsequently,    the   sets $\mathcal S_j$ and the universe  $\mathcal U$, constituting the minimum set covering problem,  can be constructed with linear complexity in the number of state variables in $\mathcal D(\bar A)$.  

To show correctness,  suppose, on the contrary, we  have $\mathcal J^*\subset \{1,\ldots, p\}$ such that $\{\mathcal{S}_{j}\}_{j\in\mathcal{J}^{\ast}}$ is a (minimal)  solution to the minimum set covering problem, and $B(\mathcal J^*)$ is not a (minimal) solution to the minCIS. Hence, there exists $\mathcal J^{-}\subset \{1,\ldots, p\}$, with $|\mathcal J^{-}|<|\mathcal J^*|$,  such that $\bar B(\mathcal J^{-})$ is a solution to minCIS. Now note that since $\bar{A}$ satisfies Assumption~1, the bipartite graph $\mathcal{B}(\bar{A})$ consists of a perfect matching, and hence, by Proposition~\ref{designPerfectMatching}, for each non-top linked SCC $\mathcal{N}^{i}$ of $\mathcal{D}(\bar{A})$, there exists an input 
corresponding to an index in $\mathcal{J}^{-}$ that is assigned to a state variable in $\mathcal{N}^{i}$. 

 Thus,  by construction of the minimum set covering problem, the family  $\{\mathcal S_{l}\}_{l\in \mathcal J^-}$ covers $\mathcal U=\{1,\ldots, k\}$. Since   $|\mathcal J^{-}|<|\mathcal J^*|$,  it follows that  the  family $\{\mathcal{S}_{j}\}_{j\in\mathcal{J}^{\ast}}$  is not a minimal set covering of $\mathcal U$, a contradiction.
\end{proof}

In the next section, we introduce a class of multi-agent networked control problems, referred to as leader-selection problems. Further,  we explain how  the reduction obtained in Theorem~\ref{mainresult2} can be used to solve these  leader-selection problems.

\section{Illustrative Example}

{

To illustrate the results established in Section~4, we introduce two (structural) variants of leader-selection problems  stated in~\cite{Rahmani}, namely, (i) the  structural  (unconstrained) leader-selection, and (ii) the structural  constrained leader-selection, as presented next  in $\bar{\mathcal L}_1$ and $\bar{\mathcal L}_2$ respectively. 
We will also show that although the proposed method to solve both problems requires the solution of a set covering problem, problem $\bar{\mathcal L}_1$  is considerably easier to solve than $\bar{\mathcal L}_2$; more precisely, although the
set covering problem is in general dificult to solve, the class of problems in $\bar{\mathcal L}_1$ and the associated instances of the minimum set covering problems  can be solved by resorting to polynomial algorithms.

%

The \emph{structural  (unconstrained) leader-selection} problem can be posed as follows:
Consider a multi-agent network consisting of $N$ agents, where each agent $i$ has the ability to transmit scalar data to its neighbors and perform updates given by a linear combination of the states it receives as well its own.  Let $\bar W\in\{0,\star \}^{N\times N}$ denote the sparsity induced by such linear combination rules, and $\mathbb{I}_N=\text{diag}(\star,\ldots, \star)\in\{0,\star \}^{N\times N}$ a structural pattern of a  diagonal matrix without zeros on it;  further, we assume that $\bar W$ has nonzero diagonal entries. In addition, let each agent be equipped with an input that only actuates directly  its own state, i.e., a \emph{dedicated input}, which can be represented by letting the input matrix to be $\mathbb{I}_N$. The structural  (unconstrained) leader-selection problem aims to determining the minimum collection of agents that are required to use their inputs to ensure structural controllability. Formally, we have the following problem:

$\bar{\mathcal L}_1$ Determine $ \mathcal J^*$ where
\begin{align}
 \mathcal J^*=\arg\min\limits_{\mathcal J\subset \{1,\ldots,N\}}& \qquad\qquad |\mathcal J| \label{leaderSel} \\
\text{s.t. }  (\bar W& ,\mathbb{I}_N(\mathcal J)) \text{ is structurally controllable. }\notag
\end{align}

Alternatively, in the \emph{structural  (constrained) leader-selection} problem, we can consider similar dynamics structure $\bar W'\in\{0,\star \}^{N\times N}$ (assumed with non-zero diagonal entries), but instead of considering that each agent is equipped with a dedicated input, we assume that they receive input signals from external entities. These entities, can be understood as leaders  labelled as $\mathcal{L}=\{1,\cdots,L \}$, corresponding to the set of $L$ potential leaders  whose goal is to control the collection of $N$ followers, in this case the agents. Furthermore, denote by $\bar B\in\mathbb{R}^{N\times L}\in\{0,\star\}$ the structure of the input matrix representing the actuation exercised by  the potential leader agents, i.e., the entry $B_{f,l}$ indicates how leader $l\in \mathcal L$ actuates the follower $f \in\{1,\cdots,N\}$. Finally, given a subset $\mathcal{J}\subset\mathcal{L}$, $\bar B(\mathcal{J})$ denotes the collection of columns in $\bar B$ corresponding to indices in $\mathcal{J}$. The structural  (constrained) leader-selection problem can be posed as follows:

$\bar{\mathcal L}_2$ Determine $ \mathcal J^*$ where
\begin{align}
 \mathcal J^*=\arg\min\limits_{\mathcal J\subset \mathcal L}& \qquad\qquad\qquad |\mathcal J|\label{constLeaderSel} \\
\text{s.t. }\quad & (\bar W',\bar{B}(\mathcal{J})) \text{ is structurally controllable. }\notag
\end{align}

}

We now show that $\bar{\mathcal L}_1,\bar{\mathcal L}_2$ can be solved using  set covering problems by employing the reduction developed in Theorem \ref{mainresult2}. 

\begin{proposition}
The structural dynamics matrices $\bar W,\bar W'\in \{0,\star\}^{N\times N}$ associated with the leader-selection problems $\bar{\mathcal{L}}_1$,  $\bar{\mathcal{L}}_2$ satisfy Assumption 1.\hfill $\diamond$
\label{multiagents}
\end{proposition}

\begin{proof}
Let $\bar A\in\{0,\star\}^{N\times N}$ denote the structural matrix $\bar W$ or $\bar W'$ (depending on which problem we consider). The  proof follows by noticing that  $\mathcal D(\bar A)$  consists of  self-loops on all the state vertices, corresponding to the nonzero diagonal entries in $\bar{A}$. Consequently,  the matching  $M^*=\{ (x_i,x_i)$, $i=1,\ldots, n\}$ is  a maximum matching associated with the state bipartite graph $\mathcal B(\bar A)$, which is  a perfect matching. In other words,  the set of right-unmatched vertices of $\mathcal{B}(\bar{A})$  is empty, and hence Assumption 1 holds.
\end{proof}

Because Assumption 1 holds for the problems $\bar{\mathcal L}_1$ and $\bar{\mathcal L}_2$, by invoking Theorem \ref{mainresult2}, it follows that  we can  solve the structural  leader-selection problems  using a minimum set covering problem. 

\begin{corollary}
The problems $\bar{\mathcal L}_1,\bar{\mathcal L}_2$ can be polynomially reduced to minimum set covering problems as given in Theorem~\ref{mainresult2}. 
\hfill $\diamond$
\label{multiagentSetcover}
\end{corollary}

Now, consider the system state digraphs depicted in Figure~\ref{multiagents}. The agent states are depicted by black vertices (labeled as $x_i$, $i=1,\ldots,9$), and  the inter-agent dynamical coupling by the black directed edges. Furthermore, consider potential input vertices depicted by blue vertices (labeled as $u_i$, $i=1,\ldots,4$), where we have the following two cases: in Figure~\ref{multiagents} a) we pose the structural unconstrained leader-selection problem, whereas, in Figure~\ref{multiagents} b) , we consider a structural constrained leader-selection problem, in which  the blue directed edges (from the inputs to the agents' states) represent which leaders can actuate which agents.

Hereafter, we illustrate how, both  the structural leader-selection problems can be solved using the polynomial reduction developed in Theorem~\ref{mainresult2}  (see also Corollary 3).

\subsubsection*{Structural (Unconstrained) Leader Selection Problem} 

The goal is  to solve the leader-selection problem $\bar{\mathcal{L}}_{1}$ as formulated in \eqref{leaderSel} with the structure of the dynamics matrix induced by the state digraph represented by the black vertices and edges as depicted in Figure~\ref{multiagents} a). To this end, note that, by Proposition 1 and Corollary 3,  $\bar{\mathcal{L}}_{1}$ can be reduced to  a set covering problem (see  Theorem~\ref{mainresult2}). From Theorem \ref{mainresult2}, to set up the set covering problem, we obtain $\mathcal S_{l}=\emptyset$ for $l\in\{1,\ldots, 9\}$ since none of the (potential)  inputs $u_1,\ldots, u_9$, i.e. the dedicated inputs assigned to agents 1 to 9 respectively,    are  assigned to  variables in  non-top linked SCCs. In addition, $\mathcal S_{10}=\{1\}$, $\mathcal S_{11}=\{2\}$, $\mathcal S_{12}=\{3\}$, $ \mathcal S_{13}=\{4\}$  , where each set comprises the index of the non-top linked SCC it belongs to, and  subsequently the universe $\mathcal U=\{1,2,3,4\}$. It is readily seen that  the solution  to the set covering problem is unique and comprises the sets $\mathcal S_{l'}$, with $l'\in \{10,11,12, 13\}$. Hence, from the viewpoint of leader-selection, agents 10 to 13 should be designated as leaders, which uniquely solves the leader-selection problem. Thus, an input must  be assigned to the state variables $x_{l'}$ ($l'\in \{10,11,12,  13\}$), as depicted in Figure \ref{multiagents} a) by the blue vertices.  It is important to note that in general the set covering problems resulting from structural unconstrained leader-selection problems have the characteristic that the sets $\mathcal{S}_{l}$'s  comprise at most a single state variable. It is readily seen that such instances of the set covering problem may be solved using polynomial complexity algorithms (recall the set covering problem is NP-complete in general); in fact, to cover the universe, we only need to consider a  set for each of the elements in the universe. This is in accordance with the fact that \eqref{minDedInputSel} can be solved using a polynomial complexity algorithm (see Theorem~\ref{subclassCMISPoly}).

\begin{figure}[t]
\centering
\includegraphics[scale=0.25]{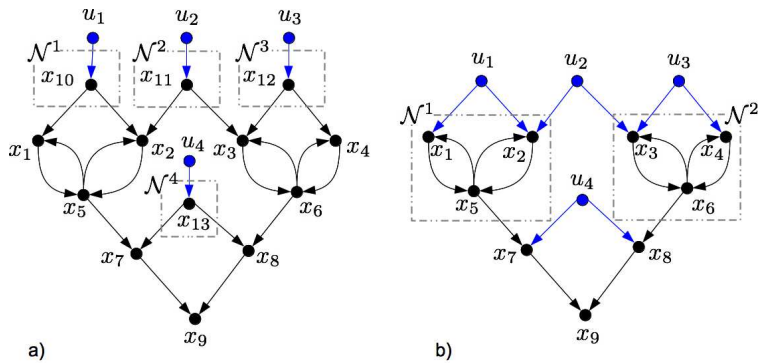}
\caption{The non-top linked SCCs are depicted by gray dashed boxes and  all agents have self-loops (not drawn to keep the illustration simple). In a) we depict the inter-agent communication graph  (the agents are depicted by black vertices with associated states as labels) given by the black edges. In addition potential leaders   (the vertices $u_{1}, u_{2}, u_{3}$ depicted in  blue) are shown to which a dedicated input may be assigned. Alternatively, in b) we depict a communication graph and possible locations for leaders (the vertices $u_{1}, u_{2}, u_{3}$ depicted in blue). }
\label{multiagents}
\end{figure}

\subsubsection*{Structural Constrained Leader Selection Problem}

Now consider  the constrained leader-selection problem $\bar{\mathcal{L}}_{2}$ as formulated in \eqref{constLeaderSel}, with the state digraph induced by the structural dynamics matrix given by  the black vertices and edges as depicted in Figure~\ref{multiagents} b) and the set of potential leaders depicted by the blue vertices. Additionally, the set of followers actuated by the potential leaders is depicted by the blue edges, i.e.,   $\bar B\in\{0,\star\}^{9\times 4}$ with all entries equal to zero except:  $\bar B_{1,1}=\bar B_{2,1}=\star$ corresponding to input $u_1$ assigned to state variables $x_1$ and $x_2$ respectively and, similarly, $\bar B_{2,2}=\bar B_{3,2}=\star$, $\bar B_{3,3}=\bar B_{4,3}=\star$, $\bar B_{7,4}=\bar B_{8,4}=\star$. Now note that, by Proposition 1 and Corollary 3, $\bar{\mathcal{L}}_{2}$ can be reduced to a set covering problem (see Theorem~\ref{mainresult2}).  From Theorem~\ref{mainresult2}, to set up the set covering problem, we obtain $\mathcal S_1=\{1\}$, $\mathcal S_2=\{1,2\}$, $\mathcal S_3=\{2\}$ and $\mathcal S_4=\emptyset$. In other words, agent 1 can only actuate followers from the non-top linked SCC $\mathcal N^1$, agent 2 can actuate followers from the non-top linked SCCs $\mathcal N^1, \mathcal N^2$ and so on. Additionally, the universe  is  $\mathcal U=\{1,2\}$ and  in this particular example  (note that in general the minimum  set covering problem is NP-hard), it is straightforward to see that the solution of  the set covering problem consists of the set $\mathcal S_2$ only. Thus agent 2 should be designated as the leader, which is the solution to the structural constrained leader-selection problem.

\section{Conclusions and Further Research}\label{conclusions}

In this paper, we have showed that the decision version of the minimum constrained  input selection (minCIS) problem is NP-complete; hence, the minCIS is NP-hard. Consequently,  in general, efficient (polynomial complexity) solution procedures to the minCIS are unlikely to exist. Nevertheless, we have identified one  subclass  of problems, of interest for control systems applications, where the minCIS is efficiently solvable, namely, minCIS instances with dedicated inputs, which can be solved polynomially. The NP-completeness of the decision version of the minCIS further implies that it is polynomially reducible to other NP-complete problems. Subsequently, for a restricted subclass of minCIS problems, which subsumes practically relevant multi-agent networked control applications such as leader-selection problems, we have explicitly constructed a polynomial reduction from the minCIS to the minimum set covering problem.   As future research, it may be worthwhile to obtain reductions from more general instances of the minCIS to other standard NP-hard problems, notably the ones with good approximation guarantees, such as the MAX-SAT -- the optimization version of the SAT problem~\cite{Garey:1979:CIG:578533}.

\section*{Appendix}
{
To prove Theorem~\ref{subclassCMISPoly}, we first introduce and review some of the results presented in  \cite{PequitoJournal,PequitoACC13}. More precisely, consider the \emph{minimal structural controllability problem} stated as follows: Given $\bar A\in \{0,\star\}^{n\times n}$, determine $\bar B^* $ such that
\begin{align}
\bar B^*=\arg\min\limits_{\bar B^* \in \{0,\star\}^{n\times n}}& \qquad\qquad \|\bar B\|_0 \label{minDedInputSelSparse}\\
\text{s.t. }\quad & (\bar A,\bar B) \text{ is  structurally controllable } \notag\\
& \|\bar B_{.,j} \|_0\le 1,\quad j=1,\ldots, n,\notag
\end{align}
where $\bar B_{.,j}$ corresponds to the $j$-th columns of $\bar B$ and $\|M\|_0$ counts the number of nonzero entries in the matrix $M\in\{0,\star\}^{n_1\times n_2}$.  

The problem \eqref{minDedInputSelSparse} (in fact, a more general variant of \eqref{minDedInputSelSparse}) was shown to be polynomially solvable in \cite{PequitoJournal,PequitoACC13}, from which we readily conclude that the minimum dedicated input selection (and output selection, by duality) is polynomially solvable.  
Further, we note that the sparsity minimization objective (as in \eqref{minDedInputSelSparse}) is not generally equivalent to the minCIS, which is consistent with the fact that the minCIS general instance is NP-hard, whereas, the sparsest input/output design problems addressed in \cite{PequitoJournal} are polynomially solvable. Nevertheless, we can use~\eqref{minDedInputSelSparse} to prove Theorem~\ref{subclassCMISPoly} as follows.

\textit{Proof of Theorem~\ref{subclassCMISPoly}}: The proof follows by noticing that a solution to~\eqref{minDedInputSelSparse}, is of the form $\bar B^*=[\mathbb{I_n^{\mathcal J}} \ \mathbf{0}_{n\times (n-|\mathcal J|)}]$ (up to permutation), where  $\mathbb{I}_n^{\mathcal J}$ corresponds to the columns of $\mathbb{I}_n$ with indices in $\mathcal J$, and $\mathbf{0}_{n\times (n-|\mathcal J|)}$ is the $n\times (n-|\mathcal J|)$ matrix of zeros. Further, we have that $\|\bar B^*\|_0=|\mathcal J|$, and since $\bar B^*$ is a solution to \eqref{minDedInputSelSparse}, it follows that $|\mathcal J|$ is minimum. Consequently, $(\bar A,\mathbb{I}_n^{\mathcal J})$ in \eqref{minDedInputSelSparse} is structurally controllable, and it readily follows that $(\bar A,\mathbb{I}_n(\mathcal J))$ in \eqref{minDedInputSel} is structurally controllable. Because, by definition, $\mathbb{I}_n^{\mathcal J}$ in~\eqref{minDedInputSelSparse}  is the same as $\mathbb{I}_n(\mathcal J)$ in \eqref{minDedInputSel}, the minimality in the latter holds. Hence, from a minimal solution to~\eqref{minDedInputSelSparse}, it is possible to retrieve a minimal solution to~\eqref{minDedInputSel}.

}

\bibliographystyle{plain}        
\bibliography{automatica13}           

\begin{thebibliography}{10}

\bibitem{CommaultD13}
Christian Commault and Jean-Michel Dion.
\newblock Input addition and leader selection for the controllability of
  graph-based systems.
\newblock {\em Automatica}, 49(11):3322--3328, 2013.

\bibitem{Coo71}
Stephen~A. Cook.
\newblock The complexity of theorem-proving procedures.
\newblock In {\em Proceedings of the third annual ACM symposium on Theory of
  computing}, STOC '71, pages 151--158, New York, NY, USA, 1971. ACM.

\bibitem{Cormen}
Thomas~H. Cormen, Clifford Stein, Ronald~L. Rivest, and Charles~E. Leiserson.
\newblock {\em Introduction to Algorithms}.
\newblock McGraw-Hill Higher Education, 2nd edition, 2001.

\bibitem{dionSurvey}
Jean-Michel Dion, Christian Commault, and Jacob~Van der Woude.
\newblock Generic properties and control of linear structured systems: a
  survey.
\newblock {\em Automatica}, pages 1125--1144, 2003.

\bibitem{Magnus}
Magnus Egerstedt.
\newblock {Complex networks: Degrees of control}.
\newblock {\em Nature}, 473(7346):158--159, May 2011.

\bibitem{Garey:1979:CIG:578533}
Michael~R. Garey and David~S. Johnson.
\newblock {\em Computers and Intractability: A Guide to the Theory of
  NP-Completeness}.
\newblock W. H. Freeman \& Co., New York, NY, USA, 1979.

\bibitem{IlicLe2008}
M.D. Ilic, L.~Xie, U.A. Khan, and J.M.F. Moura.
\newblock Modeling of future cyber physical energy systems for distributed
  sensing and control.
\newblock {\em IEEE Transactions onSystems, Man and Cybernetics, Part A:
  Systems and Humans}, 40(4):825 --838, July 2010.

\bibitem{MesbahiEgerstedt}
Mehran Mesbahi and Magnus Egerstedt.
\newblock {\em {Graph theoretic methods in multiagent networks}}.
\newblock Princeton University Press, 2010.

\bibitem{rezatac07}
Reza Olfati-Saber, J.~Alex Fax, and Richard~M. Murray.
\newblock {Consensus and Cooperation in Networked Multi-Agent Systems}.
\newblock {\em Proceedings of the IEEE}, 95(1):215--233, January 2007.

\bibitem{PequitoJournal}
S.~Pequito, S.~Kar, and A.~P. Aguiar.
\newblock A framework for structural input/output and control configuration
  selection of large-scale systems.
\newblock {\em Submitted to IEEE Trans. on Automatic Control, Available in
  {http://arxiv.org/abs/1309.5868}}, 2013.

\bibitem{PequitoACC13}
S.~Pequito, S.~Kar, and AP. Aguiar.
\newblock A structured systems approach for optimal actuator-sensor placement
  in linear time-invariant systems.
\newblock In {\em American Control Conference (ACC), 2013}, pages 6108--6113,
  June 2013.

\bibitem{Rahmani}
Amirreza Rahmani, Meng Ji, Mehran Mesbahi, and Magnus Egerstedt.
\newblock Controllability of multi-agent systems from a graph-theoretic
  perspective.
\newblock {\em SIAM J. Control Optim.}, 48(1):162--186, February 2009.

\bibitem{largeScale}
Dragoslav~D. Siljak.
\newblock {\em Large-Scale Dynamic Systems: Stability and Structure}.
\newblock Dover Publications, 2007.

\bibitem{Skogestad04a}
Sigurd Skogestad.
\newblock Control structure design for complete chemical plants.
\newblock {\em Computers and Chemical Engineering}, 28(1-2):219--234, 2004.

\end{thebibliography}

\end{document}